\documentclass{article}

\usepackage{amsmath}
\usepackage{amssymb}
\usepackage{amsfonts}
\usepackage{amscd}

\usepackage{amsthm}

\usepackage{enumerate}
\usepackage{graphicx}
\usepackage{subfig}
\usepackage{hyperref}
\usepackage{url}

\def \a {\alpha}
\def \A {\mathcal{A}}
\def \B {\mathcal{B}}
\def \N {\mathbb{N}}

\def \E {\mathcal{E}}
\def \Z {\mathbb{Z}}
\def \R {\mathbb{R}}
\def \M {\mathcal{M}}
\def \maj{\mbox{{\tt maj}}}
\def \traf{\mbox{{\tt traf}}}
\def \gkl{\mbox{{\tt gkl}}}

\newcommand{\Neighb}{\mathcal{N}}
\newcommand{\zero}{{\bf 0}}
\newcommand{\one}{{\bf 1}}
\newcommand{\Dzero}{{\delta_{\zero}}}
\newcommand{\Done}{{\delta_{\one}}}
\newcommand{\T}{\mathcal{T}}

\newtheorem{proposition}{Proposition}
\newtheorem{theorem}{Theorem}
\newtheorem{conjecture}{Conjecture}

\newcommand{\putfig}[2]{
\includegraphics[scale=#1]{fig/#2}
}

\newcommand{\FigIPS}{
\begin{figure}
\begin{center}
\putfig{0.7}{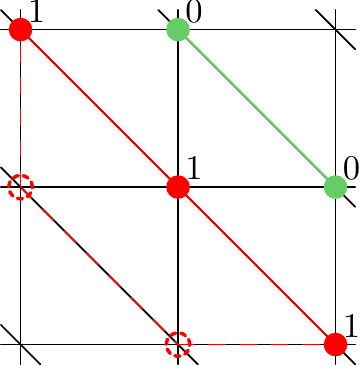} \quad
\putfig{0.7}{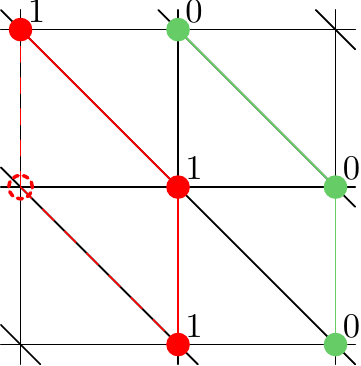} \quad
\putfig{0.7}{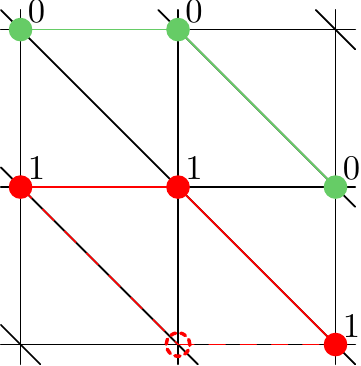}
\end{center}
\caption{Illustration of the definition of the IPS.}
\label{fig:IPS}
\end{figure}
}

\newcommand{\FigTheorem}{
\begin{figure}
 \begin{center}
 \subfloat[]{\putfig{0.7}{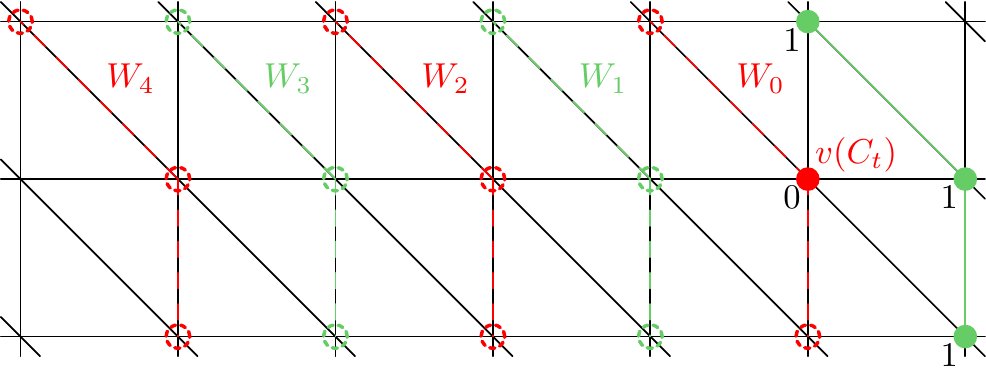}} \hfill
 \subfloat[]{\putfig{0.6}{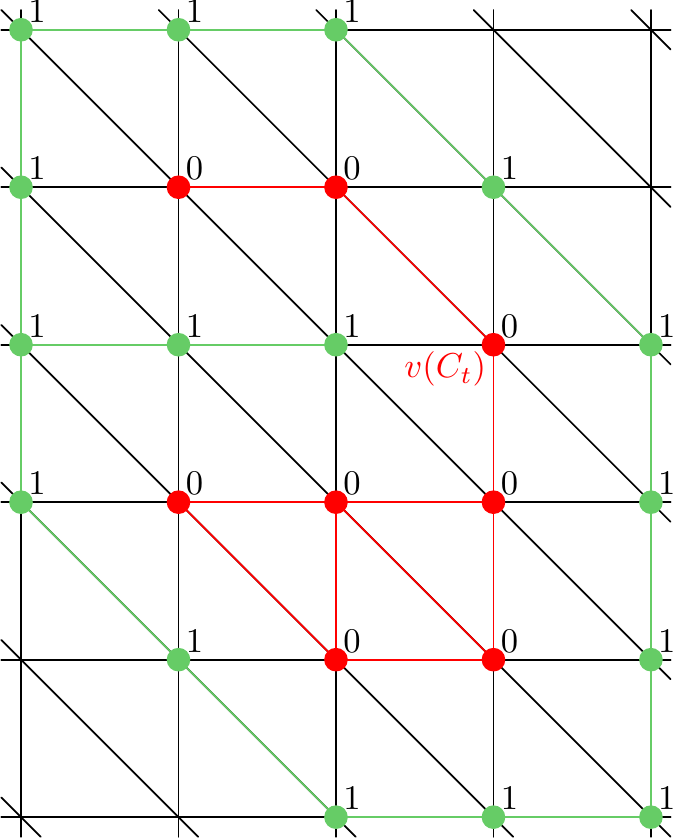}}
 \end{center}
 \caption{Illustration of the proof of Theorem \ref{thm:ips}}\label{fig:ips}
\label{fig:theorem} 
\end{figure}
}

\newcommand{\FigCAtheo}{
\begin{figure}
\begin{center}
\subfloat[]{\putfig{0.45}{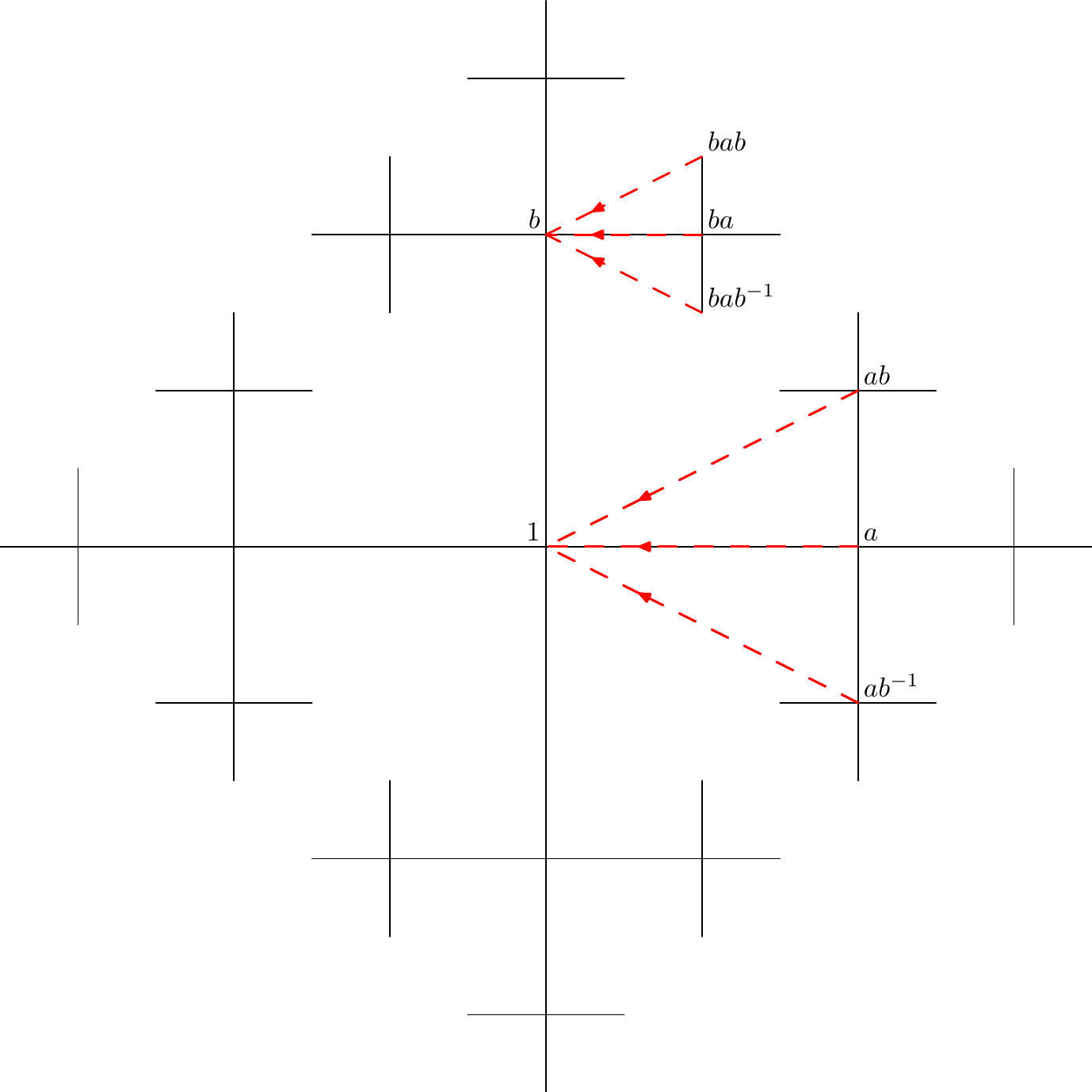}}
\subfloat[]{\putfig{0.45}{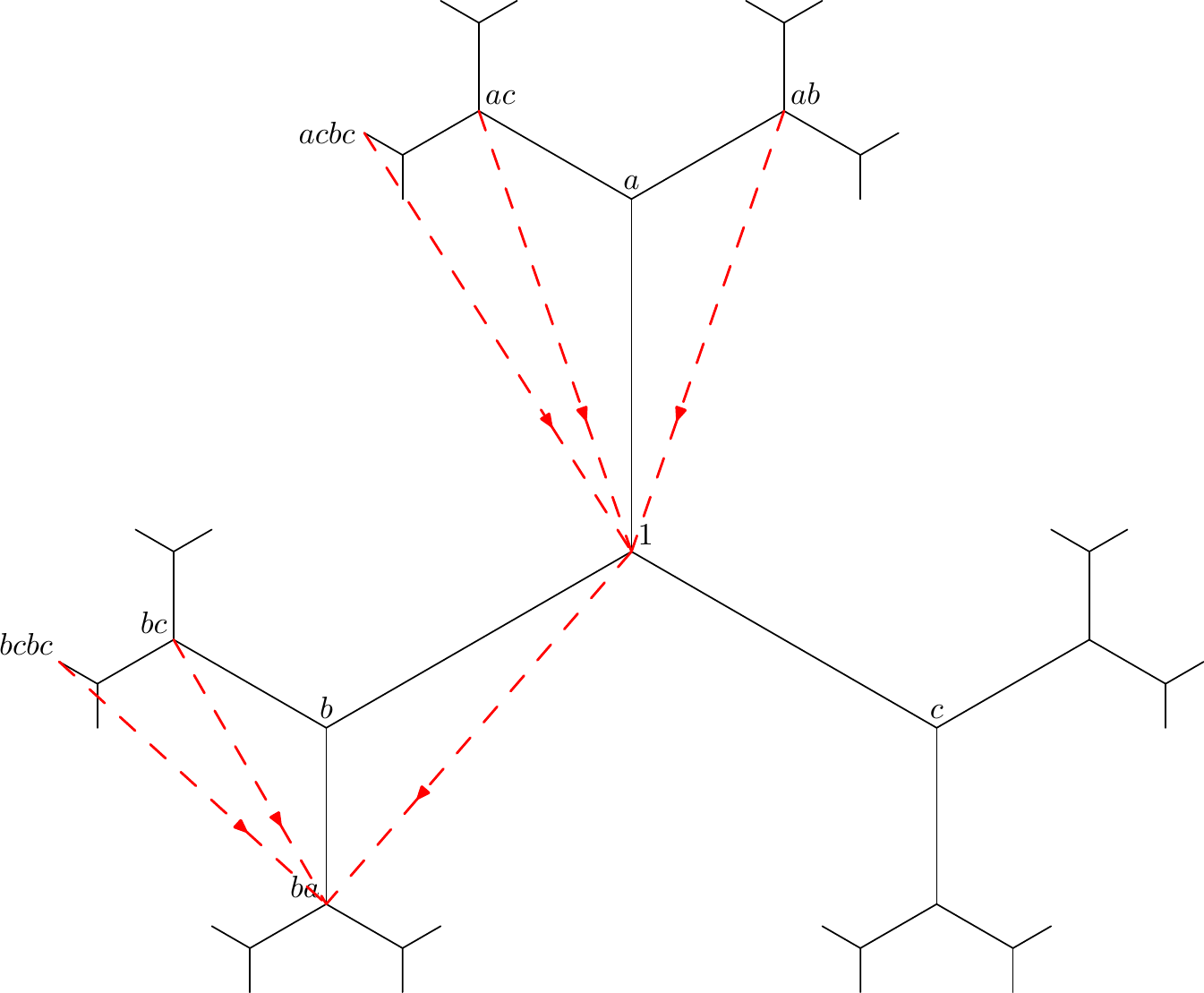}}
\end{center}
\caption{The cellular automata described by Theorem \ref{thm:T4} and Theorem \ref{thm:T3}}
\label{fig:CAtheo}
\end{figure}
}

\newcommand{\FigDiagTraMaj}{
\begin{figure}
\begin{center}
\putfig{0.25}{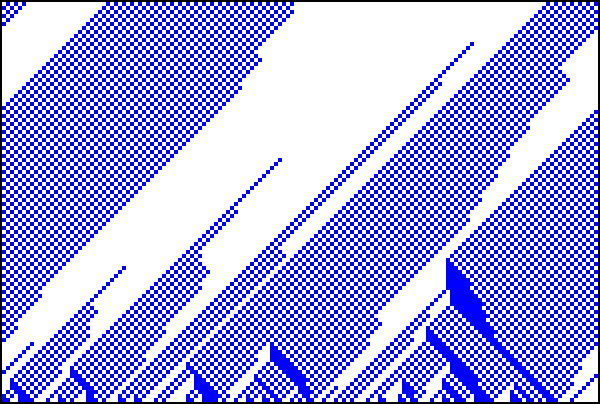}
\quad
\putfig{0.25}{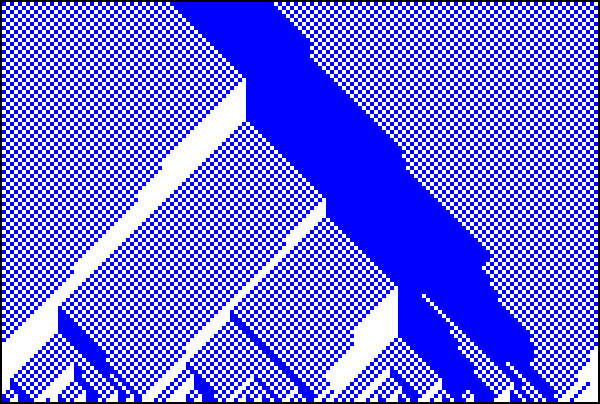}

\end{center}
\caption{Two space-time diagrams of the majority-traffic PCA for $\alpha=0.1$ and $ n =149 $. The {\em same} initial condition with density 70/149 is used. The case seen on the right is a rare event (evolution towards a bad classification).}
\label{fig:TraMaj}
\end{figure}
}

\newcommand{\FigDiagGKLKari}{
\begin{figure}
\begin{center}
\begin{tabular}{c c}
\putfig{0.25}{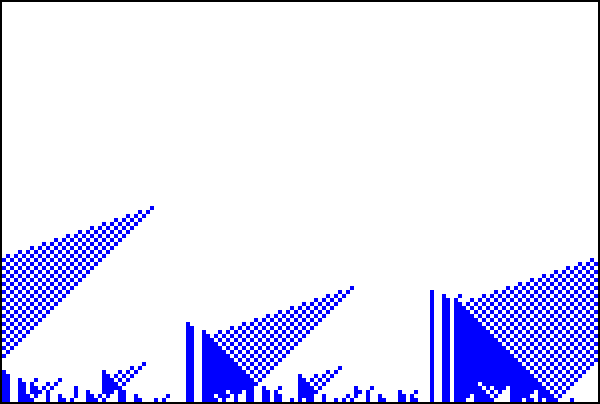}
&
\putfig{0.25}{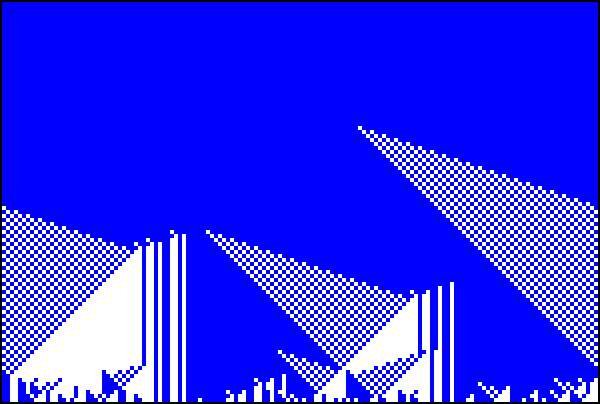}\\
GKL, $ d < 1/2$ & 
GKL, $ d > 1/2$ \\
\end{tabular}

\bigskip
\begin{tabular}{c c}
\putfig{0.25}{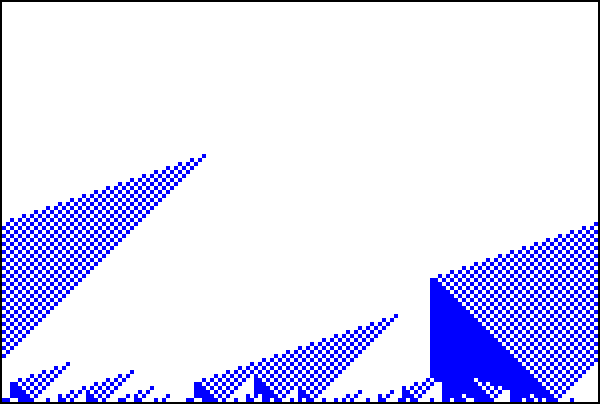}
&
\putfig{0.25}{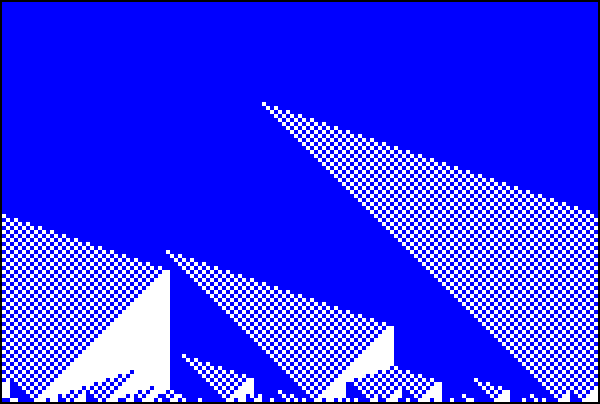}\\
Kari, $ d < 1/2$ & 
Kari, $ d > 1/2$ \\
\end{tabular}

\end{center}
\caption{Two space-time diagrams of GKL (top) and Kari's PCA (bottom) for $n=149$.
Initial condition with density 70/149 (left) and 77/149 (right).}
\label{fig:Kari}
\end{figure}
}

\newcommand{\FigDistrib}{
\begin{figure}
\begin{center}
\putfig{0.45}{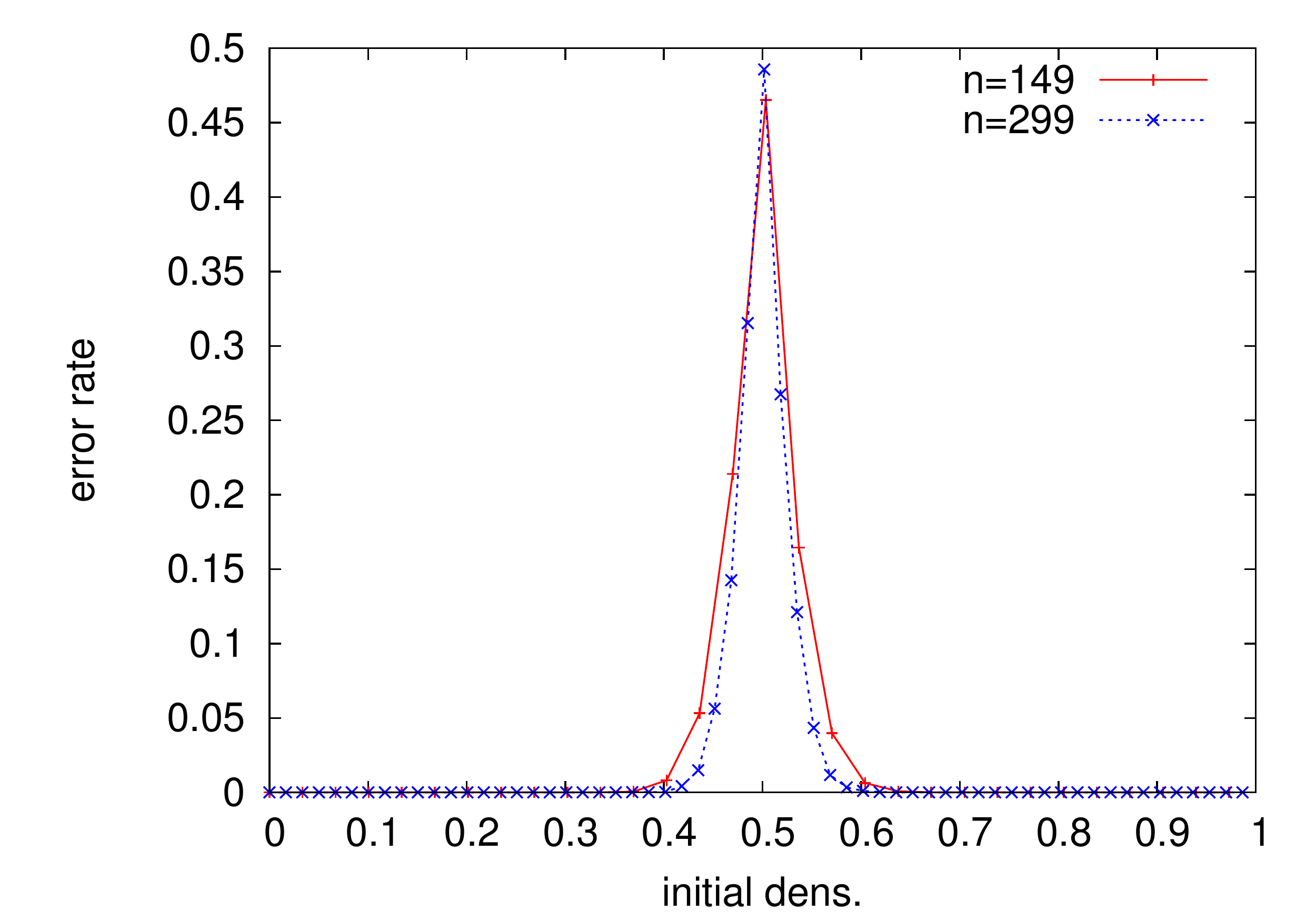}
\end{center}
\caption{Majority-traffic rule with $ \alpha=0.25$: Evolution of the error rate as a function of the initial density when doubling the ring size.}
\label{fig:distrib}
\end{figure}
}

\newcommand{\FigScalingLaws}{
\begin{figure}
\begin{center}
\begin{tabular}{c}
\putfig{0.33}{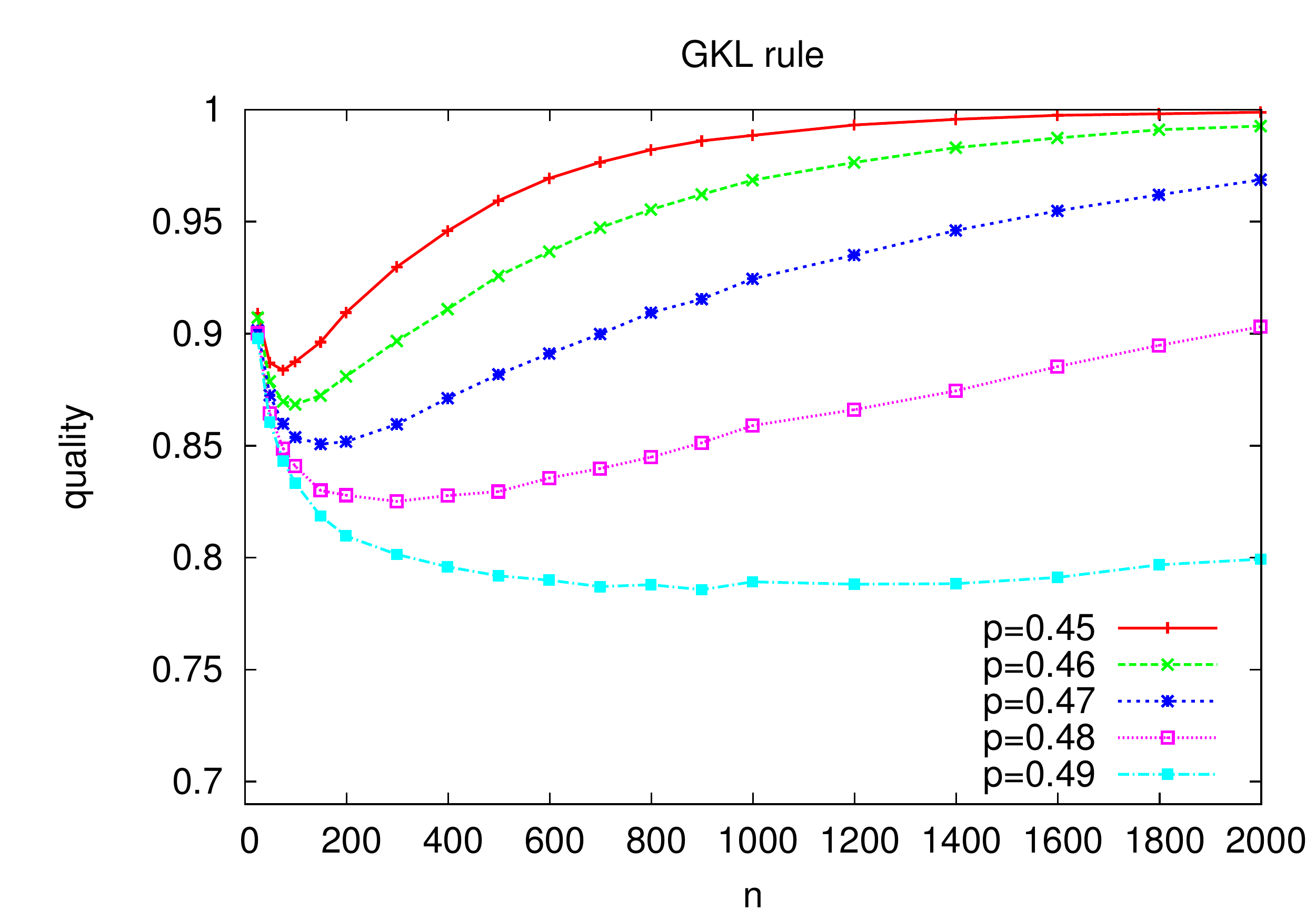} \\
\putfig{0.33}{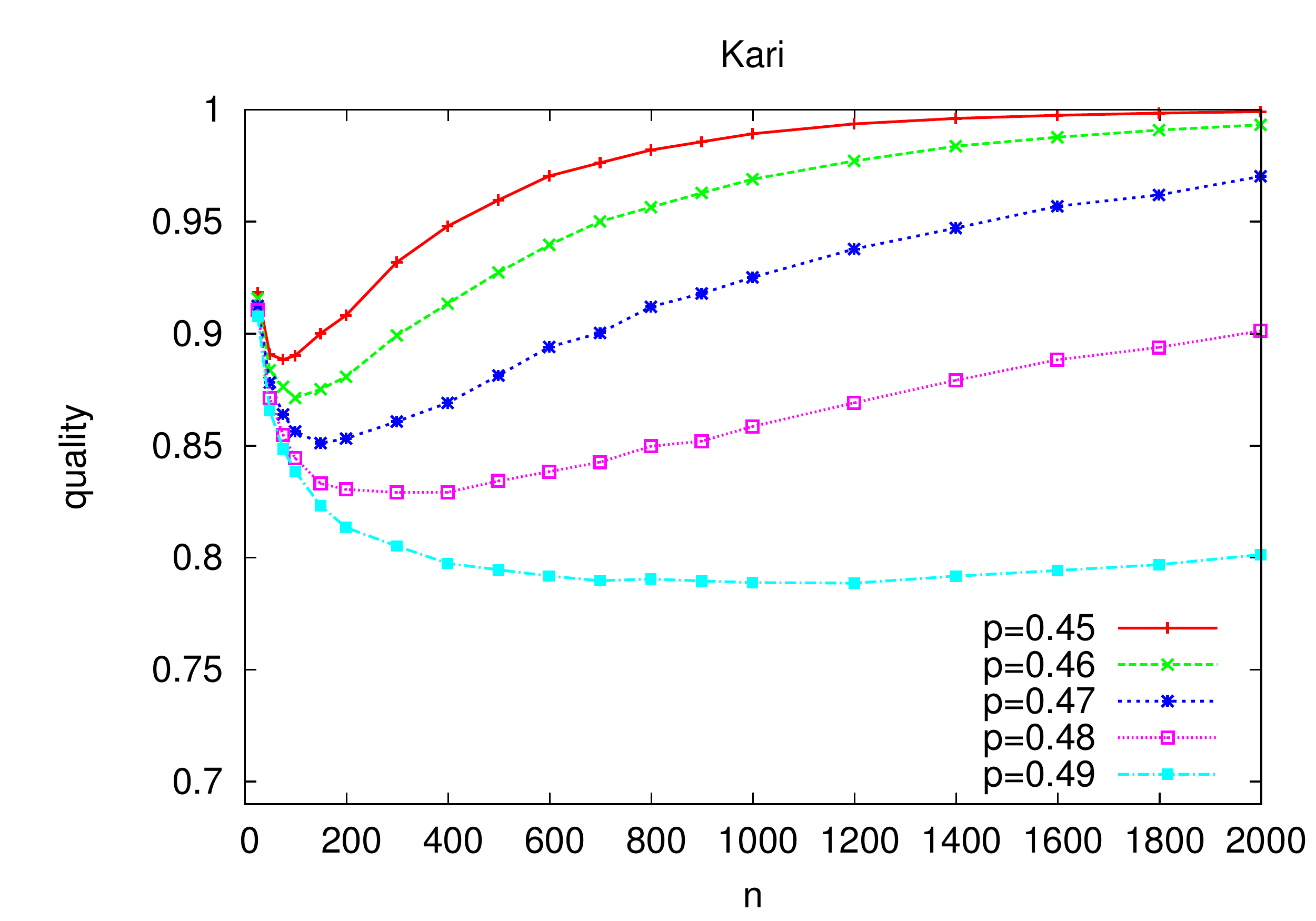}\\
\putfig{0.33}{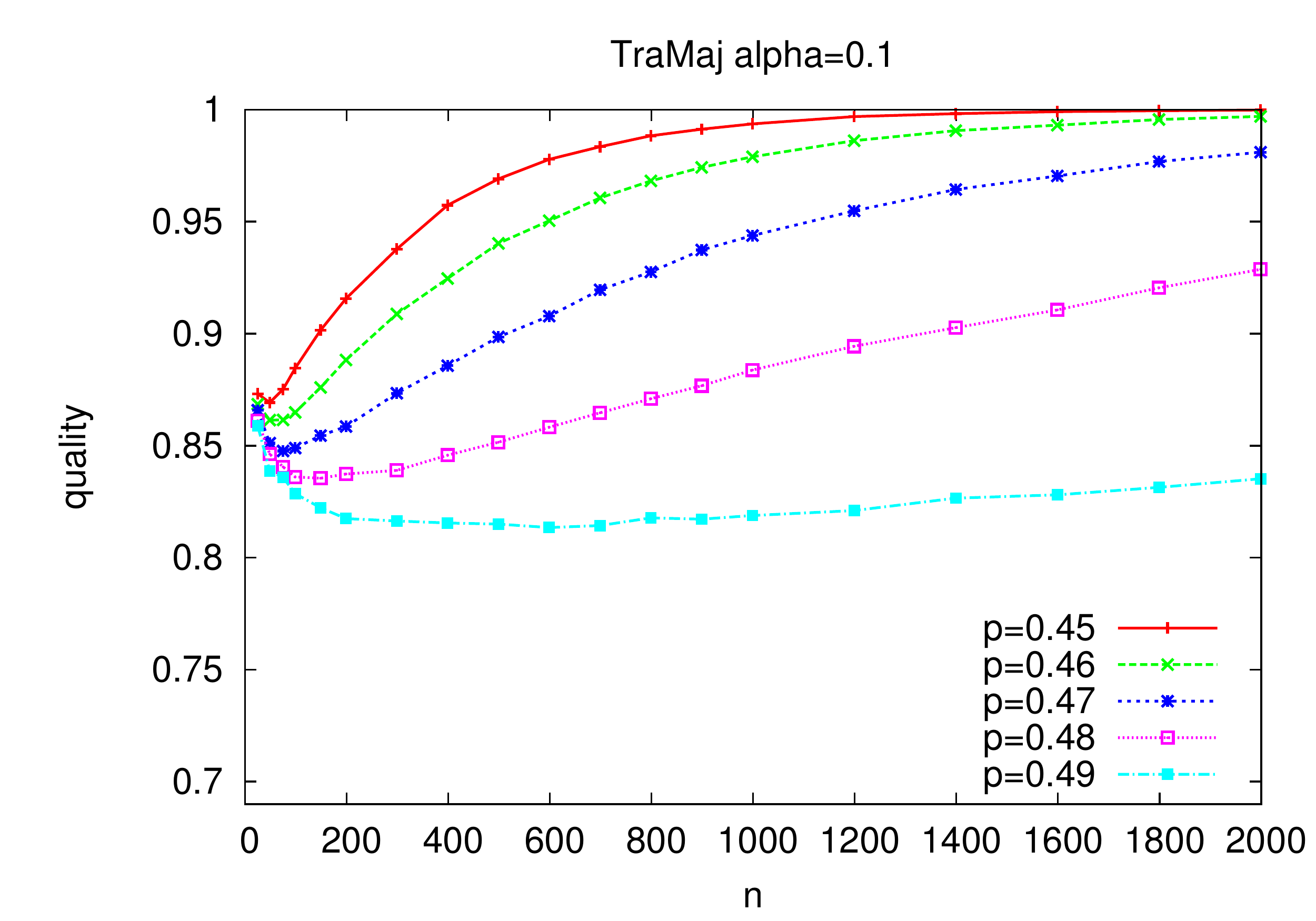} \\
\end{tabular}
\end{center}
\caption{Experimental determination of the quality of classification $ Q(n) $ as a function of ring size $ n $. Cells are initialised with a probability $p$ to be in state~$1$. Each point represents an average computed out $100\,000$ trajectories.}
\label{fig:scalingQuality}
\end{figure}
}

\begin{document}

\title{Density classification on infinite \\ lattices and trees}%

\author{Ana {\sc Bu\v{s}i\'c}%
\thanks{INRIA/ENS, 23, avenue d'Italie, CS 81321, 75214 Paris Cedex 13, France. 
E-mail: {\tt Ana.Busic@inria.fr}.}
\and
Nazim {\sc Fat\`es}
\thanks{INRIA Nancy -- Grand-Est, LORIA, Nancy Universit\'e, BP 239, 54506, Vand\oe uvre-l\`es-Nancy, France.
E-mail: {\tt Nazim.Fates@loria.fr.}}
\and
Jean {\sc Mairesse}%
\thanks{LIAFA, CNRS and Universit\'e Paris Diderot - Paris 7, Case 7014, 75205 Paris Cedex 13, France.
E-mail: {\tt Jean.Mairesse@liafa.jussieu.fr}.}
\and 
Ir\`ene {\sc Marcovici}%
\thanks{LIAFA, CNRS and Universit\'e Paris Diderot - Paris 7, Case 7014, 75205 Paris Cedex 13, France. E-mail: {\tt Irene.Marcovici@liafa.jussieu.fr} (corresponding author).}}

%
%
%
%

\maketitle              

\begin{abstract}
Consider an infinite graph with nodes initially labeled by independent Bernoulli
random variables of parameter $p$. We address the density classification problem, that is, we want to design a (probabilistic or deterministic)
cellular automaton or a finite-range interacting particle system that evolves on this graph and decides whether 
$p$ is smaller or larger than $1/2$. Precisely, the trajectories should 
converge
to the uniform configuration with only $0's$ if $p<1/2$, and only $1's$ if $p>1/2$.  We present solutions
to that problem on $\Z^d$, for any $d\geq 2$, and on the regular
infinite trees. For $\Z$, we propose some candidates that we
back up with numerical simulations. 

\end{abstract}

\section{Introduction}

Consider a finite or a countably infinite set of {\em cells}, which are spatially arranged according to a group structure $ G $.
We are interested in the {\em density classification} problem, which consists of deciding in a decentralised way, if an initial configuration on $ G $ contains more 0's or more 1's. 
More precisely, the goal is to design a deterministic or
probabilistic dynamical system that evolves in the {\em configuration space} $  \{0,1\}^G$ with a local and homogeneous updating rule
and whose trajectories
converge to $0^G$ or to $1^G$ if the initial
configuration contains more 0's or more 1's, respectively. 
To attack the problem, two natural instantiations of dynamical
systems are considered, one with
synchronous updates of the cells, and one with asynchronous
updates.  In the first case, time is
discrete, all cells are updated at each time step, and the model is known as a  {\em
  Probabilistic Cellular Automaton (PCA)}~\cite{toom}. A {\em Cellular Automaton
  (CA)} is a PCA in which the updating rule is deterministic. In the
second case, 
time is continuous, cells are updated at random instants, at most
one cell is updated at any given time, and the model is known as a (finite range) {\em
  Interacting Particle System (IPS)}~\cite{liggett}. 

\smallskip

The general spirit of the problem is that of distributed computing: gathering a
global information by exchanging only local information.
The challenge is
two-fold: first, it is impossible to centralise the information
(cells are indistinguishable); 
second, it is impossible to use classical counting techniques (cells contain
only a binary information). 

\smallskip

The density classification problem was originally introduced for rings of
finite size ($G=\Z/n\Z$) and for synchronous models~\cite{packard}. 
After experimentally observing that finding good rules to perform this task was difficult, it was shown that perfect classification with CA is impossible, that is, there exists no given
CA that solves the density classification problem for all values of
$n$~\cite{land}. This result however did not stop the quest for the
best -- although imperfect -- models as nothing was known about how
well CA could perform. The use of PCA opened a new path
\cite{fuks,schuele} and it was shown that there exist PCA that can
solve the problem with an arbitrary precision~\cite{fates}.  In the
present paper, Prop. \ref{pr-pca}, we complement the results from
\cite{land,fates} by showing that there exists no PCA that solves the density classification problem for all values of
$n$. 

\smallskip

The challenge is now to extend the research to infinite groups (whose
Cayley graphs are lattices or regular trees). First, we need to specify the meaning of
``having more 0's or more 1's'' in this context. 
Consider a random configuration on $\{0,1\}^G$ obtained by assigning
independently to each cell a value 1 with probability $p$ and a value
0 with probability $1-p$. 
We say that a model ``classifies the density''
if the trajectories converge
weakly to $1^G$ for $p>1/2$, and to $0^G$ for $p<1/2$. 
A couple of conjectures and negative results exist in the literature. 
Density classification on $\Z^d$ is considered in \cite{cox}
under the name of ``bifurcation''. 
The authors study variants of the famous voter model
IPS~\cite[Ch. V]{liggett} and they propose two instances that are
conjectured to bifurcate. 
The density classification question has also been
addressed for the Glauber dynamic associated to the Ising
model at temperature 0, both for lattices and for
trees~\cite{fontes,howard,kanoria}.  The Glauber dynamic defines an IPS or PCA having $0^G$ and $1^G$ as
invariant measures. 
Depending on the cases, there is either a proof that the Glauber
dynamic does not classify the density, or a conjecture that it
does with a proof only for densities sufficiently close to $0$ or
$1$. 

\smallskip

The density classification problem has been approached with
different perspectives on finite and infinite groups, as emphasised by
the results collected above. For finite
groups, the problem is studied {\em per se}, as
a benchmark for understanding the power and limitations of PCA as a
computational model. The community involved is rather on the computer
science side. For infinite groups, the goal is to understand the
dynamics of specific models that are relevant in statistical
mechanics. The community involved is rather on the theoretical physics
and probability theory side. 

\smallskip

The aim of the present paper is to investigate how to generalise the
finite group approach to the infinite group case. 
We want to build models of PCA and IPS, as simple
as possible, that correct random noise in the initial configuration, even if the density of errors
is close to $1/2$. 
We consider the groups $\Z^d$, whose Cayley graphs are lattices (Section \ref{se:Z^2}), and
the free groups, whose Cayley graphs are infinite regular trees (Section \ref{se:trees}). 
In all cases, except for $\Z$, we obtain both PCA and IPS models
that classify the density. To the best of our
knowledge, they constitute the first known such examples. 
The case of $\Z$ is more complicated and could be linked to the
so-called {\em positive rates conjecture}~\cite{gacs_pos}. We provide some
potential candidates for
density classification together with simulation experiments (Section
\ref{se:Z}).

\section{Defining the density classification problem}

Let $(G,\cdot)$ be a finite or countable set of \emph{cells} equipped
with a group structure. Set $\A=\{0,1\}$, the {\em alphabet}, and
$X=\A^G$, the set of {\em configurations}. For $x\in X$ and $u\in
\{0,1\}$, denote by $|x|_u$ the number of occurences of $u$ in $x$. 

\subsection{PCA and IPS}

%

Given a finite set $\Neighb \subset G$, a \emph{transition function} 
of \emph{neighbourhood} $\Neighb$
is a function $f: \A^{\Neighb} \rightarrow  \A$. 
The \emph{cellular automaton (CA)} $F$ of transition
function $f$ is the application $F: X \rightarrow  X $
defined by:
$$\forall x\in X, \forall g\in G, \quad F(x)_g=f((x_{g\cdot v})_{v\in \Neighb}).$$
When the group $G$ is $\Z^d$ or $\Z_n=\Z/n\Z$, we denote as usual the law of $G$ by the sign $+$, so that the definition can be written:
$\forall x\in X, \forall k\in \Z^d \mbox{ (resp. } \Z_n\mbox{)},
F(x)_k=f((x_{k+v})_{v\in \Neighb}).$

\emph{Probabilistic cellular automata (PCA)} are an extension of
classical CA: the transition function is now a function $ \varphi :
\A^{\Neighb} \rightarrow  \M(\A)$, where $\M(\A)$ denotes the set of
probability measures on $\A$. At each time step, the cells are updated
synchronously and independently, according to a distribution depending
on a finite neighbourhood \cite{toom}. This defines an application
$F:\M(X)\rightarrow \M(X)$. The image of a measure $\mu$ is denoted by
$\mu F$. 

\smallskip

The analog of PCA in continuous time are 
\emph{(finite-range) interacting particle systems (IPS)}~\cite{liggett}.  
IPS are characterised by a finite neighbourhood
$\Neighb \subset G$, and a transition function $f: \A^{\Neighb}
\rightarrow  \A$ (or $\varphi: \A^{\Neighb} \rightarrow  \M(\A)$). We
attach random and independent clocks to the cells of
$G$. For a given cell, the instants of $\R_+$ at which the clock rings form a
Poisson process of parameter 1. 
Let $x^t$ be the configuration at time
$t\geq 0$ of the process. If
the clock at cell $g$ rings at instant~$t$, the state of the cell $g$ is
updated into $f((x^t_{g\cdot v})_{v\in\Neighb})$ (or according to the
probability measure $\varphi((x^t_{g\cdot v})_{v\in\Neighb})$). This defines a
transition semigroup $F=(F^t)_{t\in\R_+}$, with $F^t:\M(X)\rightarrow
\M(X)$. If the initial measure is $\mu$, the distribution of the process at
time $t$ is given by $\mu F^t$. 

\smallskip

In a PCA, all cells are updated at each time step, in a
``synchronous'' way. On the other hand, for an IPS, the updating is
``asynchronous''. Indeed, the probability of having
two clocks ringing at the same instant is 0. 

\smallskip

Observe that PCA are discrete-time Markov
chains, while IPS are continuous-time Markov
processes. A measure $\mu$ is said to be an \emph{invariant mesasure}
of a process $F$, resp. $(F_t)_t$, if $\mu F=\mu$, resp.  $\mu F_t=\mu$ for all $t\in \R_+$.

\subsection{The density classification problem on $\Z_n$}

The density classification problem was originally stated as follows:
find a finite neighbourhood $\Neighb\subset\Z$ and a transition
function $f:\A^{\Neighb}\rightarrow \A$ such that for any integer
$n\geq 1$ and any configuration $x\in\A^{\Z_n}$, when applying the CA
$F$ of transition function $f$ to $x$, the sequence of iterates
$(F^k(x))_{k\geq 0}$ reaches the fixed point $\zero=0^n$ if
$|x|_0> |x|_1 $ and the fixed point $\one=1^n$ if $|x|_1> |x|_0$.
Land and Belew~\cite{land} have proved that there exists no CA that perfectly performs
 this density classification task for all values of $n$. 
We now prove that this negative result can be extended to the PCA. It provides at the same time a new proof for CA as a particular case. 

\smallskip

Denote by $\delta_x$ the probability measure that corresponds to a Dirac
distribution centred on $x$. 

\begin{proposition}\label{pr-pca} There exists no PCA or IPS that solves perfectly the
  density classification problem on $\Z_n$,  that is, for any
  integer $n\geq 1$, and for any configuration $x\in\A^{\Z_n}$, 
   $(\delta_x F^t)_{t\geq 0}$ converges to $\delta_{\zero}$  if
   $|x|_0>n/2$ and to  $\delta_{\one}$ if $|x|_1>n/2$.  
\end{proposition}

\begin{proof} We carry out the proof for PCA. For IPS, the argument is
  similar and even simpler. 
Let us assume that $F$ is a PCA that solves perfectly
  the density classification problem on $\Z_n$.
Let $\Neighb$ be the
  neighbourhood of $F$, and let $\ell$ be such that $\Neighb\subset
  (-\ell,\ell)$. We will prove that for any $x\in\A^{\Z_n}$ (with
  $n\geq 2\ell$), the number of occurrences of $1$ after application of
  $F$ to $x$ is almost surely equal to $|x|_1$. Let us assume that it
  is not the case. Then, we have:
\begin{equation}\label{eq-diff}
\exists x,y \in\A^{\Z_n}, \ |x|_1 \neq |y|_1, \qquad \delta_x F (y) >0 \:.
\end{equation}
Assume for instance that $|y|_1>|x|_1$ (the case $|y|_1<|x|_1$ is
treated similarly). 
We first assume that $|x|_1=a > n/2$. For some integers $k\geq2,m\geq
2\ell$, let us consider the configuration
$z=x^k0^{m}\in\A^{\Z_{kn+m}}$. We have $|z|_1=ka$. Let $y_s$ be the
suffix of length $n-\ell$ of $y$, and let $y_p$ be the prefix of
length $n-\ell$ of $y$. 
By applying 
(\ref{eq-diff}), it follows that:
\[
\exists u,v,u',v' \in \A^{\ell}, \qquad \delta_z
F (uy_sy^{k-2}y_pvu'0^{m-2\ell}v') >0 \:.
\]
Set $w=uy_sy^{k-2}y_pvu'0^{m-2\ell}v'$. We have $|w|_1 \geq k|y|_1 -2\ell \geq k(a+1)
- 2\ell$. 
For $m$ big enough, if we set $k$ to be the largest integer such that $k(a-n/2)<m/2$, we have:
\[
|z|_1=ka < {kn+m\over 2}, \qquad |w|_1 \geq k(a+1) -2\ell>
{kn+m\over 2} \:.
\]
So, with a positive probability, we can reach a configuration
with more ones than zeros starting from a configuration with more
zeros than ones. Since $F$ classifies the
density with probability $1$, the new configuration can be considered as an initial condition that needs to be classified and will thus almost surely evolve to the fixed point $\one$, that is, a bad classification will occur, which contradicts our hypothesis. 

The case $|x|_1=a < n/2$ is analogous, except that we now consider
configurations $z$ of the form  $x^k1^{m}\in\A^{\Z_{kn+m}}$ and choose
the integers $k, m$ such that $ka+m< (kn+m)/2$ and
$k(a+1)+m-2\ell > (kn+m)/2$.  

We have proved that for any $x\in\A^{\Z_n}$ (with $n\geq \ell$), the
number of occurrences of ones after application of $F$ to $x$ is almost
surely equal to $|x|_1$. This is in contradiction with the fact that
$F$ classifies the density. 
\end{proof}

This proposition can be extended to larger dimensions: for any $d\geq
1$, there is no PCA or IPS that classifies perfectly the density on all the
groups of the form $\Z_{n_1}\times\ldots\times\Z_{n_d}$. 

\subsection{The density classification problem on infinite groups}\label{se:pb}

Let us define formally the density classification problem on infinite
groups.  

We denote by $\mu_p$ the Bernoulli measure of parameter $p$, that is,
the product measure of density $p$ on $X=\A^G$. A realisation of
$\mu_p$ is obtained by assigning independently
to each element of $G$ a label $1$ with probability $p$ and a label
$0$ with probability $1-p$. We denote respectively by $\zero$ and
$\one$ the two uniform configurations $0^G$ and $1^G$ and by $
\delta_x $ the probability measure that corresponds to a Dirac
distribution centred on $x$. 

\smallskip

The {\em density classification problem} is
to find a PCA or an IPS $F$, such that:  
\begin{equation}\label{def_pb}
\begin{cases}
p<1/2  \implies \mu_pF^t \xrightarrow[t\rightarrow \infty]{w} \Dzero \\
p>1/2  \implies \mu_pF^t \xrightarrow[t\rightarrow \infty]{w} \Done \\
\end{cases}.
\end{equation}

The notation $\xrightarrow{w}$ stands for the weak convergence of measures. In our
case, the interpretation is that for any {\em finite} subset $K\subset
G$, the probability that at time $t$, all the cells of $K$ are labelled
by $0$ (resp. by $1$) tends to $1$ if $p<1/2$ (resp. if $p>1/2$). Or
equivalently, that for any single cell, the probability that it is
labelled by $0$ (resp. by $1$) tends to $1$ if $p<1/2$ (resp. if
$p>1/2$). 

\subsection{From subgroups to groups}

Next proposition has 
the following consequence: given a process that classifies the density
on $\Z^2$, we can design a new one that classifies on $\Z^d$ for $ d >
2$. The idea is to divide $\Z^d$ into $\Z^2$-layers
and to apply the original process independently on each layer. 

\begin{proposition}\label{prop:group} Let $H$ be a subgroup of $G$,
  and let $F_H$ be a process (PCA or IPS) of neighbourhood $\Neighb$
  and transition function $f$ that classifies the density on
  $\A^H$. We denote by $F_G$ the process on $\A^G$ having the same
  neighbourhood $\Neighb$ and the same transition function $f$. Then,
  $F_G$ classifies the density on $\A^G$. 
\end{proposition}

\begin{proof} Since $H$ is a subgroup, the group $G$ is partitioned into a union of classes
  $g_1H, g_2H,\ldots$ We have $\Neighb\subset H$, so that if an element
  $g\in G$ is in some class $g_iH$, then for any $v\in\Neighb$,
  the element $g\cdot v$ is also in $ g_iH$. Since $F_H$ classifies the density, on each
  class $g_iH$, the process $F_G$ satisfies (\ref{def_pb}). 
Thus for any cell of $G$, the probability that it is labelled by $0$
(resp. by $1$) tends to $1$ if $p<1/2$ (resp. if $p>1/2$). 
 \end{proof}

\section{Classifying the density on $\Z^2$: Toom's rule}\label{se:Z^2}

To classify the density on $\Z^2$, a natural idea is to apply
the majority rule on a cell and its four direct neighbours. Unfortunately, 
this does not work, neither in the CA nor in the IPS version. Indeed, an
elementary square of four cells in state 1 on a background of 0's is a
fixed point for the process. For $p\in (0,1)$, monochromatic
elementary squares of both colors
appear almost surely in the initial configuration which makes the convergence to
$\zero$ or $\one$ impossible. 

Another idea is to apply the majority rule on the four nearest neighbours
(excluding the cell itself) and to choose uniformly the new state of
the cell in case of equality. 
In the IPS setting, this process is known as the
Glauber dynamics associated to the Ising model. It has been
conjectured to classify the density, but the result has been proved
only for values of $p$ that are sufficiently close to $0$ or $1$~\cite{fontes}. 

\smallskip

To overcome the difficulty, we consider the majority CA but on the asymmetric  neighbourhood $\Neighb=\{(0,0),(0,1),(1,0)\}$.
We prove that this CA, known as Toom's rule \cite{toom,gacs},
 classifies the density on $\Z^2$.
Our proof relies on the properties of the percolation clusters on the triangular lattice~\cite{grimmett}. 
We then define an IPS inspired by this local rule and prove with the same techniques that it also classifies the density. 

\subsection{A cellular automaton that classifies the density}

Let us denote by $\maj:\A^3\rightarrow \A$, the majority function, so that 
$\maj(x,y,z)=0 \mbox{ if } x+y+z<2$ and $1 \mbox{ if } x+y+z\geq 2.$

\begin{theorem}\label{thm:toom}The cellular automaton $\T:\A^{\Z^2}\rightarrow\A^{\Z^2}$ defined by:
$$\T(x)_{i,j}=\maj(x_{i,j},x_{i,j+1},x_{i+1,j})$$
for any $x\in\A^{\Z^2}, (i,j)\in\Z^2,$ classifies the density.\end{theorem}

\begin{proof}By symmetry, it is sufficient to prove that if $p>1/2$, then $(\mu_p\T^n)_{n\geq 0}$ converges weakly to $\Done$.

Let us consider the triangular lattice of sites (vertices) $\Z^2$ and
bonds (edges)
$\{\{(i,j),(i,j+1)\},\{(i,j),(i+1,j)\},\{(i+1,j),(i,j+1)\},(i,j)\in\Z^2\}$. 
%
We recall that a {\em $0$-cluster} is a subset of connected sites labelled by
$0$ which is maximal for inclusion. The site percolation
threshold on the triangular lattice is equal to $1/2$ so that, for
$p>1/2$, there exists almost surely no infinite $0$-cluster~\cite{grimmett}. 
Thus, if $S_{\zero}$ denotes the set of sites labelled by $0$, the set
$S_{\zero}$ consists almost surely of a countable union
$S_{\zero}=\cup_{k\in\N}S_k$ of finite $0$-clusters. 
Moreover, the size of the $0$-clusters decays exponentially: there
exist some constants $\kappa$ and $\gamma$ such that the probability
for a given site to be part of a $0$-cluster of size larger than $n$
is smaller than $\kappa e^{-\gamma n}$, see \cite{grimmett}. 

\smallskip

Let us describe how the $0$-clusters are transformed by the action of
the CA. 
For $S\subset \Z^2$, let $1_S$ be the configuration defined by
$(1_S)_x = 1$ if $x\in S$ and $(1_S)_x = 0$ otherwise. Let $\T(S)$ be the subset $S'$ of $\Z^2$ such that
$\T(1_S)=1_{S'}$. By a simple symmetry argument, this last equality is
equivalent to $\T(1_{\Z^2\setminus S})=1_{{\Z^2\setminus S'}}$. 
We observe the following. 

\begin{itemize}
\item The rule does not break up or connect different $0$-clusters
  (proved by G\'acs~\cite[Fact~3.1]{gacs}). More precisely, if $S$
  consists of the $0$-clusters $(S_k)_k$, then the
  components of $\T(S)$ are the nonempty sets among $(\T(S_k))_k$.   


\item Any finite $0$-cluster disappears in finite time: if $S$ is a
  finite and connected subset of $\Z^2$, then there exists an integer
  $n\geq 1$ such that $\T^n(S)=\emptyset$. This is the \emph{eroder}
  property~\cite{toom}. 


\item Let us consider a $0$-cluster and a rectangle in
  which it is contained. Then the $0$-cluster always remains within
  this rectangle. More precisely, if $R$ is a rectangle set, that is,
  a set of the form $\{(x,y)\in\Z^2 \mid a_1\leq x\leq a_2, \ b_1\leq
  y\leq b_2\}$, and if $S\subset R$, then for all $n\geq 1$,
  $\T^n(S)\subset R$ (proof by induction). 
\end{itemize}

Let us now consider all the $0$-clusters for which the minimal enveloping
rectangle contains the origin $(0,0)$. By the exponential decay of the
size of the clusters, one can prove that the number of such
$0$-clusters is almost surely finite. Indeed, the probability that the
point of coordinates $(m,n)$ is a part of such a cluster is smaller than
the probability for this point to belong to a $0$-cluster of size
larger than $\max(|m|,|n|)$.  
And since 
$$\sum_{(m,n)\in\Z^2}\kappa e^{-\gamma\max(|m|,|n|)}<4\kappa \sum_{m\in\N}(me^{-\gamma m}+\sum_{n\geq m} e^{-\gamma n})< \infty,$$
we can apply the Borel-Cantelli lemma to obtain the result. Let $T_0$
be the maximum of the time needed to erase these $0$-clusters. The
random variable $T_0$ is almost surely finite, and after $T_0$ time
steps, the site $(0,0)$ will always be labelled by a $1$. As the 
argument can be generalised to any site, it ends the proof. 
\end{proof} 

We point out that Toom's CA classifies the density despite having
many different invariant measures. For 
example:
\begin{itemize}
\item Any configuration $x$ that can be decomposed into monochromatic
North-East paths (that is, $x_{i,j}=x_{i,j+1}$ or $x_{i,j}=x_{i+1,j}$ for any $i,j$) is a fixed
point and $\delta_{x}$ is an invariant measure. 
\item Let $y$ be the checkerboard configuration defined by $y_{i,j}=0$ if $i+j$ is
even and $y_{i,j}=1$ otherwise, and let $z$ be defined by
$z_{i,j}=1- y_{i,j}$. Since we have $
\T(y) = z $ and $ \T(z)=y $, the two configurations $y$ and $z$ form a
periodic orbit and $(\delta_y+\delta_z)/2$ is an
invariant measure. 
\end{itemize}

\subsection{An interacting particle system that classifies the density}

We now define an IPS for which we use the same steps as above
to prove that it classifies the density.

Note that the exact IPS analog of Toom's
rule might classify the density but the above proof does not carry over since, in
some cases, different 
$0$-clusters may merge. To overcome the difficulty, we introduce a
different IPS with a new neighbourhood of size $7$: the
cell itself and the $6$ cells that are connected to
it in the triangular lattice defined in the previous section. 

For $\alpha\in \A$, set $\bar{\alpha} =
1-\alpha$.

\begin{theorem}\label{thm:ips}
Let us consider the following IPS: for a
configuration $x\in{\A}^{\Z^2}$, we update the state of the cell
$(i,j)$ by applying the majority rule on the North-East-Centre
neighbourhood, except in the following cases (for which we keep the
state unchanged): 
\begin{enumerate}
\item $x_{i,j}=x_{i-1,j+1}=x_{i+1,j-1}={\bar x_{i,j+1}}=\bar x_{i+1,j}$ and ($x_{i,j-1}={\bar x_{i,j}}$ or \\ $x_{i-1,j}={\bar x_{i,j}}$),
\item $x_{i,j}=x_{i-1,j+1}=x_{i,j-1}=\bar x_{i,j+1}=\bar x_{i+1,j}=\bar x_{i+1,j-1}$ and $x_{i-1,j}=\bar x_{i,j}$,
\item $x_{i,j}=x_{i-1,j}=x_{i+1,j-1}=\bar x_{i,j+1}=\bar x_{i+1,j}=\bar x_{i-1,j+1}$ and $x_{i,j-1}=\bar x_{i,j}$.
\end{enumerate}
This IPS classifies the density.
\end{theorem}

The three cases for which we always keep the state unchanged are
illustrated below for the case where $x_{i,j}=1$ (central cell). In
the first case, we allow to flip the central cell if and
only if the two cells marked by a dashed circle are also labelled by
$1$. Otherwise, the updating could connect two different $0$-clusters
and break up the $1$-cluster to which the cell $(i,j)$ belongs to.
The second and third cases are analogous. 

\FigIPS

The proof is similar to the one of Theorem \ref{thm:toom} but involves some
additional technical points. 

 \begin{proof}We assume as before that $p>1/2$. Like the CA of the
   previous section, the new process that we have defined has the
   property not to break up or connect different clusters. 
 Furthermore, if we consider a $0$-cluster and the smallest rectangle
 in which it is contained, we can check again that the $0$-cluster
 will never go beyond this rectangle.  As before, we only need to prove
 that any finite $0$-cluster disappears almost surely in finite time to
 conclude the proof. 
 We consider a realisation of the trajectory of the IPS with initial density $\mu_p$. 
 We associate to any finite $0$-cluster $C\subset {\Z}^2$ the point
 $v(C)=\max\{(i,j)\in C\}$, where the order is the lexicographic order
 on the coordinates (we set $v(\emptyset)=(-\infty,-\infty)$). The
 point $v(C)$ is thus the upmost point of $C$ among its rightmost
 points. Let us consider at time $0$ some finite $0$-cluster $C_0$. We
 denote by $C_t$ the state of this cluster at time $t$. 

\smallskip

{\bf Claim.} {\em The value $v(C_t)$ is nonincreasing. Moreover,
   if $t\geq 0$ is such that $C_t\not=\emptyset$, then there exists
   almost surely a time $t'>t$ such that $v(C_{t'})<v(C_t)$.}

\smallskip

Let us prove the claim. 
 Let us denote by $x \in \A^{\Z^2}$ a configuration attained at some time $t$, and let $(i,j)=v(C_t)$. 
 By definition of $v(C_t)$, if a cell of coordinate $(i+1,j')$ is
 connected to a cell of $C_t$, then $x_{i+1,j'}=1$. Either we have also
 $x_{i+1,j'+1}=1$ and the cell $(i+1,j')$ will not flip. Or
 $x_{i+1,j'+1}=0$, but in this case, since $(i+1,j'+1)$ does not belong
 to $C_t$, $x_{i,j'+1}=1$ and the cell of $C_t$ to which is connected
 $(i+1,j')$ is necessarily $(i,j')$. So, $x_{i,j'}=0$ and
 $x_{i+1,j'-1}=1$, once again by definition of $v(C_t)$. Depending on
 the value of $x_{i+2,j'-1}$, either rule~1 or rule~2 forbids the cell $(i+1,j')$ to flip. 
 In the same way, we can prove that  if a cell of coordinate $(i,j'),
 j'>j$ is connected to $C_t$, then it is not allowed to flip. This
 proves that $v(C_t)$ is nonincreasing. 
 In order to prove the second part of the claim, we need to show that
 the cell $(i,j)$ will almost surely be flipped in finite
 time. By definition of $(i,j)=v(C_t)$, we know that
 $x_{i,j+1}=x_{i+1,j}=x_{i+1,j-1}=1$. The cell $(i,j)$ will thus be
 allowed to flip, except if $x_{i-1,j+1}=x_{i,j-1}=0$ and
 $x_{i-1,j}=1$. But in that case, the cell $(i-1,j)$ will end up
 flipping, except if $x_{i-1,j-1}=x_{i-2,j+1}=1, x_{i-2,j}=0$, and so on. Let
 $W_n=\{(i-n,j),(i-1-n,j+1),(i-n,j-1)\}$. If for each $n$, the cells of
 $W_n$ are in the state $(n \mod 2)$, then none of the cell $(i-n,j)$
 is allowed to flip (see Figure~\ref{fig:ips}.a). 
But recall now that the initial measure is $\mu_p$. There exists
 almost surely an integer $n\geq 0$ such that the initial state of the
 cell $(i-n,j)$ is {\it not}  $(n\mod 2)$. 
Let $m(t)$ be the smallest
 integer $n$ whose value at time $t$ is not $n \mod 2$. Then, one can easily check that $m(t)$
 is non-increasing, and that it reaches $0$ in finite time. Thus, the cell $(i,j)$ ends up flipping and we have proved the
 claim.  
 
\FigTheorem

 \smallskip

 The example of Figure~\ref{fig:ips}.b illustrates how the proof
 works. Here, no cell of the cluster $C_t$ is allowed to flip, but
 since the cells on the right and on the top of $v(C_t)$ cannot flip
 either, $v(C_t)$ does not increase. The cell at the left of $v(C_t)$
 will end up flipping, and $v(C_t)$ will then be allowed to flip. 

 \smallskip

 Since we know that a $0$-cluster cannot go beyond its enveloping
 rectangle, a direct consequence of the claim is that any $0$-cluster
 disappears in finite time. This allows us to conclude the proof in the
 same way as for the majority cellular automaton. 
 \end{proof} 

\section{Classifying the density on regular trees}\label{se:trees}

Consider the finitely presented  group $T_n=\langle a_1,\ldots,a_n \ |
\ a_i^2=1\rangle$. The Cayley graph of  $T_n$ is the infinite $n$-regular tree.
For $n=2k$, we also consider the free group with $k$ generators, that
is, $T'_{2k}=\langle a_1,\ldots,a_k \ | \ 
\cdot\rangle$. The groups $T_{2k}$ and $T'_{2k}$ are not isomorphic, but
they have the same Cayley graph. 

\subsection{Shortcomings of the nearest neighbour majority rules}

For odd values of $n$, a natural candidate for classifying the density 
is to apply the majority rule on
the $n$ neighbours of a cell. But it is proved that neither the
CA (see  \cite{kanoria} for $n=3,5,$ and 7) nor the IPS (see
\cite{howard} for $n=3$) classify the density. 


For $n=4$, a natural candidate would be to apply the majority on the four
neighbours and the cell itself. We now prove that it does not work either. 

\begin{proposition}
Consider the group $T'_4=\langle a,b \ |\ \cdot\rangle.$ Consider the
majority CA or IPS with neighbourhood $\Neighb =\{1,a,b,a^{-1}, b^{-1}\}$. 
For $p\in(1/3,2/3)$, the trajectories do not converge weakly to a
uniform configuration. 
\end{proposition}

\begin{proof}
If
$p\in(1/3,2/3)$, then we claim that at time $0$, there are almost
surely infinite chains of zeros and infinite chains of ones that are
fixed. Let us choose some cell labelled by $1$. 
Consider the (finite or infinite) subtree of
1's originating from this cell viewed as the root. If we forget the
root, the random tree is exactly
a Galton-Watson process. 
Indeed, the expected number of children of a node is  $3p$ and 
since $3p>1$, this Galton-Watson process survives with positive probability. 
Consequently, there exists almost surely an infinite chain of ones at time $0$
somewhere in the tree. In the same way, since $3(1-p)>0$, there exists
almost surely an infinite chain of zeros. 
\end{proof}

As for $\Z^2$, we get round the difficulty by keeping the majority
rule but choosing a non-symmetrical neighbourhood. 

\subsection{A rule that classifies the density on $T'_4$}

In this section, we consider the free group $T'_4=\langle a,b |\cdot\rangle,$ see Fig. \ref{fig:CAtheo} (a).

\begin{theorem}\label{thm:T4} The cellular automaton $F:\A^{T'_{4}}\rightarrow\A^{T'_{4}}$ defined by: 
$$F(x)_g=\maj(x_{ga},x_{gab},x_{gab^{-1}})$$
for any $x\in\A^{T'_{4}}, g\in T'_{4}$, classifies the density.
\end{theorem}

\begin{proof} We consider a realisation of the trajectory of the CA
  with initial distribution $\mu_p$. Let us denote by $X^n_g$ the random
  variable describing the state of the cell $g$ at time $n$. Since the
  process is homogeneous, it is sufficient to prove that $X^n_1$
  converges almost surely to $0$ if $p<1/2$ and to $1$ if $p>1/2$. 
Let us denote by $h:[0,1]\rightarrow[0,1]$ the function that maps a
given $p\in[0,1]$ to the probability $h(p)$ that $\maj(X,Y,Z)=1$ when
$X,Y,Z$ are three independent Bernoulli random variables of parameter
$p$. An easy computation provides $h(p)=3p^2-2p^3$, and one can check
that the sequence $(h^n(p))_{n\geq 0}$ converges to $0$ if $p<1/2$ and
to $1$ if $p>1/2$.  

We prove by induction on $n\in\N$ that for any $k\in\N$, the family
$\E_k(n)=\{X_{u_1u_2\ldots u_k}^n \ | \ u_1, u_2,\ldots,
u_k \in \{a,ab,ab^{-1}\}\}$ consists of independent Bernoulli random
variables of parameter $h^n(p)$. By definition of $\mu_p$, the
property is true at time $n=0$. Let us assume that it is true at some
time $n\geq 0$, and let us fix some $k\geq 0$. Two different elements
of $\E_k(n+1)$ can be written as the majority on two disjoint triples
of $\E_{k+1}(n)$. The fact that the triples are disjoint is a
consequence of the fact that $\{a, ab, ab^{-1}\}$ is a code: a given
word $g\in G$ written with the elementary patterns $a, ab, ab^{-1}$
can be decomposed in only one way as a product of such patterns. By
hypothesis, the family $\E_{k+1}(n)$ is made of i.i.d. Bernoulli
variables of parameter $h^n(p)$, so the variables of $\E_k(n+1)$ are
independent Bernoulli random variables of parameter
$h^{n+1}(p)$. Consequently, the process $F$ classifies the density on
$T'_4$.\end{proof} 

Let us mention that from time $n\geq 1$, the field $(X^n_g)_{g\in G}$
is not i.i.d. For example, $X^1_1$ and $X^1_{ab^{-1}a^{-1}}$ are not
independent since both of them depend on $X^0_a$. 

\smallskip 

On $T'_{2k}=\langle a_1,\ldots,a_k| \cdot\rangle$, one can either
apply Prop.~\ref{prop:group} to obtain a cellular automaton that
classifies the density, or define a new CA by the following formula: 
$F(x)_g=\maj(x_{ga_1},x_{ga_1a_2},x_{ga_1a_2^{-1}},\ldots,x_{ga_1a_k},x_{ga_1a_k^{-1}})$ and check that it is also classifies the density.

\smallskip

It is also possible to adapt the above proof to show that the IPS 
with the same local rule also classifies the density. 

\FigCAtheo

\subsection{A rule that classifies the density on $T_3$}

We now consider the group $T_3=\langle a, b, c \ | \ a^2=b^2=c^2=1\rangle$.

\begin{theorem}\label{thm:T3} The cellular automaton $F:\A^{T_3}\rightarrow\A^{T_3}$ defined by: 
$$F(x)_g=\maj(x_{gab},x_{gac},x_{gacbc})$$
for any $x\in\A^{T_3}, g\in T_3$, classifies the density.
\end{theorem}

\begin{proof}The proof is analogous to the previous case. We prove by
  induction on $n\in\N$ that for any $k\in\N$, that the family
  $\E_k(n)=\{X_{u_1u_2\ldots u_k}^n \ | \ u_1, u_2,\ldots,
  u_k\in\{ab,ac,acbc\}\}$ consists of independent Bernoulli random
  variables of parameter $h^n(p)$, the key point being that $\{ab, ac,
  acbc\}$ is a code.\end{proof} 

Once again, as explained in Prop.~\ref{prop:group}, since we have a
solution on $T_3$, we obtain a CA that classifies the density for any
$T_n, n\geq 3,$ by applying exactly the same rule. The corresponding IPS on $T_n$ also classifies the density. 

\section{Classifying the density on $\Z$}\label{se:Z}

The one-dimensional case appears as much more difficult than the other cases and we are not aware of any solution to the density classification problem
on~$\Z$. 
However, if we slightly change the formulation of the problem, simple
solutions do exist. We first give one such modification 
and then go back to the original problem and describe three models,
two CA and one PCA, that are conjectured to classify the density. We
also provide some preliminary analytical results as well as
experimental confirmations of these results by using numerical
simulations.

\smallskip

In the examples below, the {\em traffic} cellular automaton, rule 184
according to Wolfram's notation, plays a central role. It is the CA
with neighborhood ${\cal N}=\{-1,0,1\}$ and local function $\traf$ defined by:
$$\begin{array}{|c|c|c|c|c|c|c|c|c|}
\hline
x,y,z & 111 & 110 & 101& 100 & 011 & 010 & 001 & 000\\
\hline
\traf(x,y,z) & 1 & 0 & 1 & 1         & 1 & 0 & 0 & 0\\
\hline
\end{array}$$

This CA can be seen as a simple model of traffic flow on a single
lane: the cars are represented by 1's moving one step to the right if
and only if there are no cars directly in front of them. 
It is a density-preserving rule.

\subsection{An exact solution with weakened conditions}\label{sse-weak}

On finite rings, several models have been proposed that solve relaxed
variants of the density classification problem. We concentrate
on one of these models introduced in \cite{legloannec}. The original
setting is modified since the model operates
on an extended alphabet, and the criterium for convergence is also weakened. Modulo this relaxation, it solves the problem
on finite rings $\Z_n$. We show the same result on $\Z$. 

\begin{proposition} Consider the cellular automaton $F$ on the alphabet $\B=\A^2$, 
with neighbourhood ${\cal N}= \{-1,0,1\}$, and local
  function $f=(f_1,f_2)$ defined by:
\begin{equation}\label{eq-bastien}
f_1(x,y,z) = \traf (x_1,y_1,z_1) \quad ; \quad
f_2(x,y,z)= \begin{cases}
0 &\mbox{ if } x_1=y_1=0\\
1 &\mbox{ if } x_1=y_1=1 \\
y_2 &\mbox{ otherwise}
\end{cases}
\end{equation}
The projections $\mu_pF^n(\A^{\Z}\times \cdot)$ converge
to $\Dzero$ if $p<1/2$ and to $\Done$ if $p>1/2$.  
\end{proposition}

Intuitively, the CA operates on two tapes: on the first tape, it
simply performs the traffic rule; on the second tape, what is recorded
is the last occurrence of two consecutive zeros or ones in the first
tape. If $p<1/2$, then, on the first tape, there is convergence to
configurations which alternate between patterns of types $0^k$ and
$(10)^{\ell}$. Consequently, on the second
tape, there is convergence to the configuration $\Dzero$. 
We formalise the argument below. 

\begin{proof} Let $T: \A^{\Z}\rightarrow \A^{\Z}$ be the
  traffic CA, see above. Following an idea of Belitsky and
Ferrari~\cite{belitsky}, we
define the recoding $\psi:\A^\Z\rightarrow
\{-1,0,1\}^\Z$ by $\psi(x)_i=1-x_i-x_{i-1}$. 
Consider $(\psi \circ T^n(x))_{n\geq 0}$,
  the recodings of the trajectory of the CA originating from
  $x\in \{0,1\}^{\Z}$. 
There is a convenient alternative way to describe $(\psi
  \circ T^n(x))_{n\geq 0}$. It corresponds to the trajectories in the
  so-called \emph{Ballistic Annihilation} model: $1$ and $-1$ are interpreted as particles that we call respectively positive and negative particles. Negative particles move one cell to the left at each time step while positive particles move one cell to the right; and when two particles of different types meet, they annihilate. 

Consider the Ballistic Annihilation model with initial condition
$\mu_p \psi$ for $p>1/2$. The density of negative particles is
$p^2$, while the density of positive particles is $(1-p)^2$. During the
evolution, the density of positive particles decreases to 0, while the
density of negative particles decreases to $2p-1$. In particular, the
negative particles that will never disappear have density $2p-1$ (see \cite{belitsky} for
  details).  
We can
track back the position at time 0 of the ``eternal''
negative particles. Let $X$ be the (random) position of the first eternal
particle on the right of cell $0$. After time $X$, the column 0 in the
space-time diagram
contains only $0$ or $-1$ values. This key point is illustrated in the figure below.
\begin{center}
\putfig{0.45}{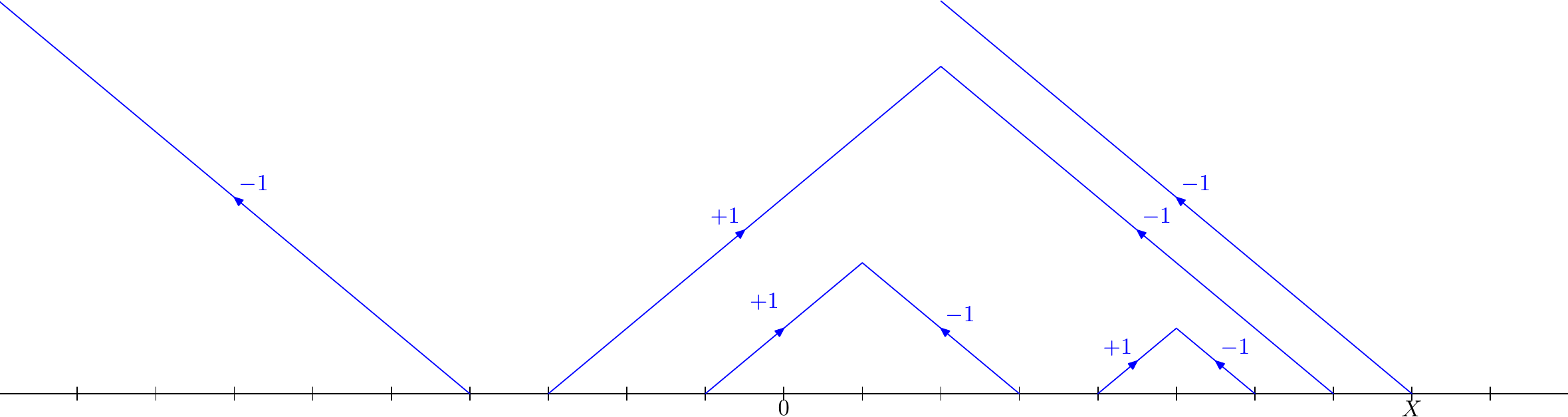}
\end{center}

We now go back to the traffic CA with initial condition
distributed according to $\mu_p$ for $p>1/2$ and concentrate on two 
consecutive columns of
the space-time diagram. The property tells us that after some
almost surely finite time, the columns contain only the patterns $11,
01$, or $10.$ 

For the CA defined by Eq. \ref{eq-bastien} with an initial  condition
distributed according to a measure $\mu$ satisfying $\mu(\cdot \times
\A^\Z)=\mu_p$ for $p>1/2$, the above key point gets translated as follows: in
any given column of the space-time diagram, after some a.s. finite
time, the column contains only the letters $(0,1)$ or $(1,1)$. In
particular, $\mu_pF^t(\A^{\Z}\times \cdot)$ converges weakly
to $\Done$ if $p>1/2$. 
\end{proof}

\subsection{Density classifier candidates on $\Z$}\label{sse-conj}
\label{se:ex}

\paragraph{The GKL cellular automaton.}

The G\'acs-Kurdyumov-Levin (GKL) cellular automaton is the CA with 
neighbourhood $\Neighb=\{-3,-1,0,1,3\}$ defined by 
$$\gkl(x)_k=
\begin{cases}
\maj(x_k,x_{k+1},x_{k+3}) &\mbox{  if  } x_k=1\\
\maj(x_k,x_{k-1},x_{k-3}) &\mbox{  if  } x_k=0.\\
\end{cases}
$$
for any $x\in\A^{\Z}, k\in\Z.$ 

The GKL CA  is known to be one
of the best performing CA for the density classification on finite
rings (see Fig.~\ref{fig:Kari}). 
It has also been proven to have the \emph{eroder} property: if
the initial configuration contains only a finite number of ones
(resp. zeros), then it reaches  $\zero$ (resp. $\one$) in finite time, see
\cite{gonzaga}. 


\paragraph{Kari traffic cellular automaton.}

This CA is defined by the composition of the two following
rules applied sequentially at each time step: (a) apply the traffic rule, (b) change the $1$ into a $0$ in every pattern $0010$ and the
  $0$ into a $1$ in every pattern $1011$ (see Fig.~\ref{fig:Kari}).

Like GKL, Kari traffic CA has a neighbourhood of radius $3$. Both CA
also share the combined symmetry consisting in swapping $0$ and $1$
and right and left. Kari
traffic has also the eroder property and it appears to have
comparable qualities to GKL concerning the density classification
task, see \cite{legloannec}. Kari traffic CA is closely related to
Kurka's modified version of GKL~\cite{kurka}.


%
%

\FigDiagGKLKari

\paragraph{The majority-traffic probabilistic cellular automaton.}

The majority-traffic PCA of parameter $\alpha\in(0,1)$ is the PCA of neighbourhood
$\Neighb=\{-1,0,1\}$ and local function:
$$\varphi(x,y,z)= \alpha \,\delta_{\maj(x,y,z)}+ (1-\alpha)\,\delta_{\traf(x,y,z)}.$$
In words, at each time step, we choose, independently for each cell, to apply
the majority rule with probability $\alpha$ and the traffic rule with
probability $1-\alpha$ (see Fig.~\ref{fig:TraMaj}).

The majority-traffic PCA has been introduced
by Fat\`es \cite{fates} who has proved that it ``classifies'' the density on a finite ring with an arbitrary precision:
for any $n\in\N$ and any $\varepsilon>0$, there exists a value
$\alpha_{n,\varepsilon}$ of the parameter such that on $\Z_n$, the PCA
converges to the right uniform configuration with probability greater
than $1-\varepsilon$.

\FigDiagTraMaj

\begin{conjecture}\label{co-classify}
The GKL CA, the Kari traffic CA, and the  majority-traffic PCA with
$0<\alpha<\alpha_c$ (for some $0<\alpha_c\leq 1/2$) classify the density.
\end{conjecture}

\subsection{Invariant Measures}

Following ideas developed by Kurka \cite{kurka}, we can give a precise description of the invariant measures
of these PCA. 

\begin{proposition}
For the majority-traffic PCA and for Kari traffic CA, the extremal
invariant measures are $\Dzero, \Done,$ and
$(\delta_{(01)^{\Z}}+\delta_{(10)^{\Z}})/2$. For GKL, on top of these
three measures, there exist extremal
invariant measures of density $p$ for any $p\in [1/3,2/3]$. 
\end{proposition}

\begin{proof}

{\bf Majority-traffic PCA.}
Let us consider the majority-traffic PCA $P$ of parameter $\a\in (0,1)$. We denote by $[x_0,\ldots,x_n]_k$ the \emph{cylinder set} of all configurations $y\in\A^{\Z}$ satisfying $y_{k+i}=x_i$ for $0\leq i\leq n$. Let $\mu$ be any shift-invariant measure. An exhaustive search shows that if at time $1$, we observe the cylinder $[100]_0$ then there are only eight possible cylinders of size $5$ at time $0$, that are: 
\begin{align*}
[01100]_{-1}, 
[10000]_{-1}, 
[10001]_{-1}, 
[10010]_{-1},\\
[10100]_{-1}, 
[11000]_{-1}, 
[11001]_{-1}, 
[11100]_{-1}.
\end{align*}
If we weight each cylinder by the probability to reach $[100]_0$ from them, we obtain the following expression:
\begin{align*}
\mu P[100]=&\;
\a(1-\a)\mu[01100]+
(1-\a)\mu[10000]+ 
(1-\a)\mu[10001]+ 
(1-\a)\mu[10010] \\
&\;+\a\mu[10100]+ 
\a^2\mu[11000]+
\a^2\mu[11001]+
\a(1-\a)\mu[11100].
\end{align*}
Since the measure $\mu$ is supposed to be shift-invariant, we do not need to specify the position of the cylinders: we denote by $\mu[x_0,\ldots,x_n]$ the value $\mu([x_0,\ldots,x_n]_k)$ which does not depend on $k\in\Z$. Gathering the terms with the same coefficient, we have:
\begin{align*}
\mu P[100]=&\;(1-\a)(\mu[100]-\mu[10011])+\a\mu[10100]+\a(1-\a)\mu[1100]+\a^2\mu[1100]\\
=&\;(1-\a)(\mu[100]-\mu[10011])+\a\mu[10100]+\a\mu[1100].
\end{align*}
Some more rearrangements provide:
\begin{align*}
\mu P[100]=&\;(1-\a)(\mu[100]-\mu[10011])+\a(\mu[100]-\mu[00100])\\
=&\;\mu[100]-(1-\a)\mu[10011]-\a\mu[00100].
\end{align*}
This proves that the sequence $(\mu P^n[100])_{n\geq 0}$ is non-increasing. Let us assume that $\mu P=\mu$. Then, $\mu[10011]=\mu[00100]=0$.

Let us consider the cylinder $[10^n0011]$ for some $n\geq 2$. If we
apply the majority rule on each cell except on the second cell from
the left, then after $n$ iterations, we reach the cylinder
$[10011]$. Since this occurs with a positive probability, we obtain
that for any $n\geq 0, \mu[10^n0011]=0$. This provides:
$\mu[0011]=\mu[00011]=\mu[000011]=\ldots=\mu[0^n11]$ for any $n\geq
2$. Consequently, $\mu[0011]=0$. From a cylinder of the form
$[00(10)^n11]$, if we choose to apply the majority rule on each cell,
then we reach the cylinder $[0011]$ in $n$ steps. Thus,
$\mu[00(10)^n11]=0$ for any $n\geq 0$. It follows that $\mu$ can be
written as the sum $\mu=\mu_0+\mu_1$ of two invariant measures, where
$\mu_0$ charges only the subshift $\Sigma_0=\{x\in\A^{\Z}\mid\forall
k\in\Z, x_kx_{k+1}\not=00\}$ and $\mu_1$ the subshift
$\Sigma_1=\{x\in\A^{\Z}\mid\forall k\in\Z, x_kx_{k+1}\not=11\}$. Let
us assume that $\mu[00]=0$ (which is the case for $\mu_0$). In the
same way that we have computed $\mu P[110]$, we can compute $\mu
P[11]$, and we obtain: 
\begin{eqnarray*}
\mu P[11]&=&\a\mu[0110]+\a\mu[1110]+\a\mu[1101]+\mu[1011]+\mu[0111]+\mu[1111]\\
&=&\a\mu[110]+\a\mu[1101]+\mu[11]-\mu[0011]\\
&=&\mu[11]+\a\mu[110]+\a\mu[1101].
\end{eqnarray*}

By hypothesis, $\mu P=\mu$, so that the last equality implies that $\mu[110]=0$. 

In all cases, if $\mu$ is a shift-invariant measure such that $\mu
P=\mu$, then $\mu[00]=\mu(\zero), \mu[11]=\mu(\one)$ and
$\mu[01]=\mu[10]=\mu((01)^{\Z})=\mu((10)^{\Z})$. 

\smallskip

{\bf Kari traffic CA.}
If at time $1$, we observe the pattern $100$ at position $0$, then, at
time $0$, that is to say before the application of Kari's 
CA, this same pattern was present at position $-1$. Indeed, one can
check that none of the cell of the pattern $100$ can have been
obtained by the transformation (b) (see the definition of Kari traffic
CA), so that one has just to consider the possible history of $100$ by
the traffic CA. In the same way, one can prove that if at time $1$, we
observe the pattern $110$ at position $0$, then, at time
$0$, this same pattern was present at position $1$. Let $\mu$ be a
shift-invariant measure such that $\mu K=\mu$, where $K$ denotes Kari
traffic CA. A consequence of the result on the patterns $100$ and
$110$ that we have just described is that $\mu K^{n+1}[110x100]=0$ for
any $n\geq 0$ and any $x\in\A^n$. But since $\mu K^{n+1}=\mu$, we
obtain $\mu[110x100]=0$ for any word $x$ on the alphabet $\A$. Once
again, we can write $\mu=\mu_0+\mu_1$ where $\mu_0$ and $\mu_1$ are
two invariant measures defined on $\Sigma_0$ and $\Sigma_1$. 

Let us consider a configuration of $\Sigma_0$, that is, without the
pattern $00$. By the traffic rule, each $0$ of the configuration will
move one cell to the left. Then by rule $1$, if a $0$ is at distance
greater than $2$ from the next $0$ on its right, it is erased by rule
(b). The result follows. 

\smallskip

{\bf GKL.} 
Any word $x\in \A^{\Z}$ that is a
 concatenation of the patterns  $u=001$ and $v=011$ 
 is a fixed point of the GKL cellular automaton: if $x_n=0$, then
 either $x_{n-1}=0$ or $x_{n-3}=0$ so that $F(x)_n=0$ and if $x_n=1$,
 then either $x_{n+1}=1$ or $x_{n+3}=1$ so that $F(x)_n=1$. As a
 consequence, GKL has extremal invariant measures of
density $p$ for any $p\in[1/3,2/3]$.
\end{proof}

%
To summarize, the majority-traffic and Kari traffic CA have a simpler set of
invariant measures. It does not rule out GKL as a candidate for
solving the density classification task, but rather indicates that it could be easier to prove the result for
majority-traffic or Kari traffic CA.


\subsection{Experimental results}

Conjecture \ref{co-classify} was first motivated by the observation of the space-time diagrams, see Fig. \ref{fig:Kari} and \ref{fig:TraMaj}. We provide some numerical results that support this conjecture.
For a given ring size $ n $, we generate an initial configuration $ x
$ by assigning to each cell the state 1 with a probability $ p $ and
the state 0 with probability $ 1- p $. Let us denote by $d(x)$ the actual density of $1$ in the configuration $x$.
We let the system evolve until
it reaches a fixed point $ \zero $ or $ \one$ and see if the fixed
point is $\zero$ for $ d(x) < 1/2 $ and $\one $ for $ d(x) > 1/2$.
The quality $ Q(n) $ corresponds to the proportion
of good classifications on a given ring of size~$ n$. 

\FigScalingLaws

Figure~\ref{fig:scalingQuality} shows the evolution of $ Q(n) $, each
value of $Q(n)$ being evaluated over 100 000 samples. For the three rules, the plots are in agreement with the hypothesis that 
the asymptotic value of $ Q(n) $ is 1. 
{From} a qualitative point of view, we observe that for all values of  $d$ the quality decreases before increasing, but this is only a border phenomenon for very small ring sizes. 
We also observe that when the initial density $d$ increases from
$0.45$ to $0.48$, the value of $ n $ needed to attain a given quality
$ Q(n) $ increases dramatically. For $ d = 0.49$, the change of
derivative of the curve $ Q(n) $ becomes hardly visible. However, our belief is that $ Q(n) $ will approach one as the lattice size grows, no matter how close $ p $ is to the critical density $ 1/2$. To see why this holds, consider the error rate $ err(d) $ obtained as the probability to make a {\em bad} classification when the initial configuration is equal to $ d $.
We experimentally observed that, as $ n $ grows, the function $ err(d) $, which is defined for values $ k/n $ with $ k \in \{0,\ldots,n \} $, approaches a Bell curve whose mean is centred on $ 1/2$ and whose tail progressively approaches the 0-axis (see Fig.~\ref{fig:distrib}). At the same time, for a fixed $ p $, the probability $ p(k/n) $ that the initial configuration has a density of ones equal to $ k/n $ follows a binomial distribution of parameter $ p $. We can thus calculate the global error rate $ E(n) = 1 - Q(n) $ with 
$ E(n) = \sum_{k=0}^n err(k/n) \cdot p(k/n)
$.
Intuitively, it can be seen that as $ n $ grows to infinity, the two distributions  $err(k/n) $ and $p(k/n)$ progressively separate as their mean value is different and their variance approaches 0. As a consequence, for larger values of $ n $,  the value $ E(n) $ progressively vanishes and the quality approaches 1.

\FigDistrib

By contrast, there are other PCA, such as the rule which was
originally studied by Fuk\'s~\cite{fuks},  and which consists in doing
a copy of the right or left neighbour with a fixed probability
$p<1/2$, and the identity otherwise. This rules preserves the density in average at each time step for finite rings  (see details in \cite{fates}). As a consequence, we have $ Q(n) = \max\, ( p, 1-p ) $ which implies that the increase of $ n $ does {\em not} improve the average performance of the system. 
In the infinite case, the preservation of the density is exact and can not allow the system to classify the density. We experimentally observed the same qualitative behaviour for the rule used by Sch\"ule~\cite{schuele} or for the Majority-Traffic rule for $ \alpha > 1/2 $. We believe that there exists a strong relationship between the asymptotic behaviour of the rule on finite rings and the ability to classify the density on $ \Z$.

\subsection{Link with the positive rates ``conjecture''}\label{se-conclu}

The difficulty of classifying the density on $\Z$ is
related to the difficulty of the ergodicity problem on $\Z$. 
By definition, a PCA or an IPS has {\em
  positive rates} if all its local probability transitions are different
from 0 and 1. In $\Z^2$, there exist positive rates PCA and IPS
that are non-ergodic (for instance, a ``positive rates version'' of
Toom's rule~\cite{toom}). 
It had been a long standing
conjecture that all positive rates PCA and IPS on $\Z$ are
ergodic. G\'{a}cs disproved the conjecture by exhibiting a
complex counter-example with several invariant measures, but with an
alphabet of cardinality
$2^{18}$ instead of~2~\cite{gacs_pos}. If we knew a process that
classifies the density on $\Z$, it could pave the way to exhibit simple examples
of positive rates processes that are non-ergodic.


\end{document}